\documentclass[12pt]{amsart}
\usepackage{amsthm}
\usepackage{amssymb}
\usepackage[table]{xcolor}
\usepackage{longtable}
\usepackage{slashbox}
\usepackage{geometry}
\usepackage{enumerate}
\usepackage{setspace}
\usepackage{pgfkeys}
\usepackage{upgreek}
\usepackage{tikz}
\usepackage{tikz-cd}
\usetikzlibrary{matrix, arrows}
\usepackage[unicode=true]
 {hyperref}
\hypersetup{
colorlinks=true,
urlcolor=black,
citecolor=blue,
linkcolor=blue,
}

\newtheorem{thm}{Theorem}[section]
\newtheorem{prop}[thm]{Proposition}
\newtheorem{lem}[thm]{Lemma}

\theoremstyle{definition}
\newtheorem{definition}[thm]{Definition}

\newtheorem{question}[thm]{Question}
\theoremstyle{remark}
\newtheorem{remark}[thm]{Remark}

\numberwithin{equation}{section}

\newcommand{\Perm}{\mathrm{Perm}}
\newcommand{\Hol}{\mathrm{Hol}}
\newcommand{\Aut}{\mathrm{Aut}}
\newcommand{\Stab}{\mathrm{Stab}}

\newcommand{\Hom}{\mathrm{Hom}}
\newcommand{\Map}{\mathrm{Map}}
\newcommand{\Inn}{\mathrm{Inn}}
\newcommand{\Out}{\mathrm{Out}}
\newcommand{\conj}{\mathrm{conj}}
\newcommand{\Reg}{\mathrm{Reg}}

\newcommand{\ff}{\mathfrak{f}}
\newcommand{\fg}{\mathfrak{g}}
\newcommand{\bZ}{\mathbb{Z}}
\newcommand{\bN}{\mathbb{N}}

\newcommand{\GL}{\mathrm{GL}}

\begin{document}

\large 

\title[Hopf-Galois structures and bijective crossed homomorphisms]{Non-existence of Hopf-Galois structures\\and bijective crossed homomorphisms}

\author{Cindy (Sin Yi) Tsang}
\address{Yau Mathematical Sciences Center, Tsinghua University}
\email{sinyitsang@math.tsinghua.edu.cn}\urladdr{http://sites.google.com/site/cindysinyitsang/} 

\date{\today}

\maketitle

\begin{abstract}
By work of C. Greither and B. Pareigis as well as N. P. Byott, the enumeration of Hopf-Galois structures on a Galois extension of fields with Galois group $G$ may be reduced to that of regular subgroups of $\Hol(N)$ isomorphic to $G$ as $N$ ranges over all groups of order $|G|$, where $\Hol(-)$ denotes the holomorph. In this paper, we shall give a description of such subgroups of $\Hol(N)$ in terms of bijective crossed homomorphisms $G\longrightarrow N$, and then use it to study two questions related to non-existence of Hopf-Galois structures. 
\end{abstract}

\tableofcontents

\newpage

\section{Introduction}

Let $G$ be a finite group, and write $\Perm(G)$ for the symmetric group of $G$. Recall that a subgroup $N$ of $\Perm(G)$ is said to be \emph{regular} if the map
\[ \xi_N: N\longrightarrow G;\hspace{1em}\xi_N(\eta) =\eta(1_G)\]
is bijective. Notice that $N$ must have the same order as $G$ in this case. There are two obvious examples, namely $\rho(G)$ and $\lambda(G)$, where
\[\begin{cases} \rho: G\longrightarrow\Perm(G);\hspace{1em}\rho(\sigma) = (\tau\mapsto \tau\sigma^{-1})\\\lambda : G\longrightarrow \Perm(G);\hspace{1em}\lambda(\sigma) = (\tau\mapsto\sigma\tau)\end{cases}\]
are the right and left regular representations of $G$, respectively. It is easy to see that  $\rho(G)$ and $\lambda(G)$ are equal precisely when $G$ is abelian. 

\vspace{1.5mm}

Now, consider a finite Galois extension $L/K$ of fields with Galois group $G$. The group ring $K[G]$ is a Hopf algebra over $K$ and its action on $L$ defines a Hopf-Galois structure on $L/K$. By C. Greither and B. Pareigis \cite{GP}, there is a bijection between Hopf-Galois structures on $L/K$ and regular subgroups of $\Perm(G)$ normalized by $\lambda(G)$, with the \emph{classical} structure $K[G]$ corresponding to $\rho(G)$. The consideration of the various Hopf-Galois structures, instead of just $K[G]$, has applications in Galois module theory; see \cite{Childs book} for a survey on this subject up to the year 2000.

\vspace{1.5mm}

Therefore, it is of interest to determine the number
\[ e(G) = \#\{\mbox{regular subgroups of $\Perm(G)$ normalized by $\lambda(G)$}\}.\]
See \cite{By96, Byott pq, Byott simple, Byott almost cyclic, Byott Childs, Byott abelian, Byott soluble, Byott squarefree, Childs simple, Childs EA, FSH, Kohl98, Z thesis} for some known results. In general, it could be difficult to compute $e(G)$ because $\Perm(G)$ might have many regular subgroups, and the papers above all make use of the following simplification due to N. P. Byott in \cite{By96}. Note that it suffices to compute
\[ e(G,N) =\# \left\{ \begin{array}{c}\mbox{regular subgroups of $\Perm(G)$ which are}
\\ \mbox{isomorphic to $N$ and normalized by $\lambda(G)$}\end{array}\right\}\]
for each group $N$ of order $|G|$. Further, define
\[ \Hol(N) = \{\pi \in \Perm(N) : \pi\mbox{ normalizes }\lambda(N)\},\]
called the \emph{holomorph} of $N$. Then, as shown in \cite{By96} or \cite[Section 7]{Childs book}, we have
\begin{equation}\label{Byott formula} e(G,N) = \frac{|\Aut(G)|}{|\Aut(N)|}\cdot \#\left\{\begin{array}{c}\mbox{regular subgroups in $\Hol(N)$}\\\mbox{which are isomorphic to $G$}\end{array}\right\},\end{equation}
which in turn may be rewritten as
\[ e(G,N) = \frac{1}{|\Aut(N)|}\cdot \#\Reg(G,\Hol(N)),\]
where we define
\[ \Reg(G,\Hol(N)) = \{\beta\in \Hom(G,\Hol(N)): \beta(G)\mbox{ is regular}\}.\]
Notice that elements of the above set are automatically injective because $N$ has order $|G|$. The number (\ref{Byott formula}) is much easier to compute because
\begin{equation}\label{Hol(N)} \Hol(N)  = \rho(N)\rtimes \Aut(N),
\end{equation}
by \cite[Proposition 7.2]{Childs book}, for example. In particular, the set $\Reg(G,\Hol(N))$ may be parametrized by certain $G$-actions on $N$ together with the bijective crossed homomorphisms associated to them; see Proposition~\ref{corr thm} below.

\vspace{1.5mm}

The purpose of this paper is to use the parametrization of $\Reg(G,\Hol(N))$ given in Proposition~\ref{corr thm} to study two questions concerning non-existence of Hopf-Galois structures; see Questions~\ref{Q1} and~\ref{Q2}.

\vspace{1.5mm}

For notation, given a group $\Gamma$, we shall write:
\begin{align*}
Z(\Gamma) & = \mbox{the center of $\Gamma$},\\
[\Gamma,\Gamma] & = \mbox{the commutator subgroup of $\Gamma$},\\
\Inn(\Gamma) & =\mbox{the inner automorphism group of $\Gamma$},\\
\Out(\Gamma) & = \mbox{the outer automorphism group of $\Gamma$}.
\end{align*}
Also, all groups considered in this paper are finite.

\subsection{Isomorphic type} In the case that $N=G$, notice that $\rho(G)$ and $\lambda(G)$ are regular subgroups of $\Hol(G)$ isomorphic to $G$. We ask: 

\begin{question}\label{Q1}Is there a regular subgroup of $\Hol(G)$ isomorphic to $G$ other than the obvious ones $\rho(G)$ and $\lambda(G)$?\end{question}

For $G$ abelian, the answer is completely known. 

\begin{thm}\label{abelian thm}Let $A$ be an abelian group. Then, we have $e(A,A)=1$ if and only if $|A| = 2^\delta p_1\cdots p_m$ for distinct odd primes $p_1,\dots,p_m$ and $\delta\in\{0,1,2\}$, where we allow the product of odd primes to be empty.
\end{thm}

Most of Theorem~\ref{abelian thm} may be deduced from \cite[Theorem 5]{Byott abelian} and results in \cite{By96, Z thesis}. In Section~\ref{abelian sec}, we shall give an alternative and independent proof of the backward implication, as well as a proof of the forward implication using only a couple results from \cite{By96, Z thesis}.

\vspace{1.5mm}

For $G$ non-abelian, the answer is not quite understood. In \cite{Childs simple}, S. Carnahan and L. N. Childs answered Question~\ref{Q1} in the negative when $G$ is non-abelian simple.  In Section~\ref{QS sec}, we shall extend their result to all quasisimple groups. Recall that $G$ is said to be \emph{quasisimple} if $G = [G,G]$ and $G/Z(G)$ is simple. 
 
\begin{thm}\label{QS thm}Let $Q$ be a quasisimple group. Then, we have $e(Q,Q)=2$.
\end{thm}

\subsection{Non-isomorphic type} In the case that $N$ has order $|G|$ but $N\not\simeq G$, there is no obvious regular subgroup of $\Hol(N)$ isomorphic to $G$. We ask:

\begin{question}\label{Q2} Is there a regular subgroup of $\Hol(N)$ isomorphic to $G$?
\end{question}

In \cite{Byott simple}, N. P. Byott answered Question~\ref{Q2} in the negative for every $N\not\simeq G$ when $G$ is non-abelian simple. One key idea in \cite{Byott simple} is the use of \emph{characteristic subgroups} of $N$, that is, the subgroups $M$ of $N$ for which $\varphi(M)=M$ for all $\varphi\in\Aut(N)$. In Section~\ref{idea sec}, we shall reformulate as well as extend this idea in terms of our Proposition~\ref{corr thm}; see Lemma~\ref{mod char} below. Then, in Section~\ref{pn sec}, we shall apply Lemma~\ref{mod char} to give an alternative proof of the following result due to T. Kohl \cite{Kohl98}.

\begin{thm}\label{cyclic thm} Let $C_{p^n}$ be a cyclic group of odd prime power order $p^n$. Then, we have $e(C_{p^n},N) = 0$ for all groups $N$ of order $p^n$ with $N\not\simeq C_{p^n}$.
\end{thm}

In view of \cite{Byott simple}, it is natural to ask whether Question~\ref{Q2} also has a negative answer for every $N\not\simeq G$ when $G$ is quasisimple. In Section~\ref{2An sec}, by applying Lemma~\ref{mod char} together with some other techniques from \cite{Byott simple}, we shall show that this is indeed the case when $G$ is in the following family.

\begin{thm}\label{2An thm}Let $2A_n$ be the double cover of the alternating group $A_n$ on $n$ letters, where $n\geq 5$. Then, we have $e(2A_n,N)=0$ for all groups $N$ of order $n!$ with $N\not\simeq 2A_n$.
\end{thm}

In order to determine $e(G)$, one has to compute $e(G,N)$ for all groups $N$ of order $|G|$. This could be difficult because there are lots of such $N$ in general. In the case that $e(G,N)=0$ for every $N\not\simeq G$, it suffices to compute $e(G,G)$ and the problem is significantly simplified. However, in most cases, we have $e(G,N)\geq 1$ for at least one $N\not\simeq G$. Nonetheless, it seems very likely that the techniques we develop in Section~\ref{app sec2} may be applied to show that $e(G,N)=0$ for a large family of $N$, whence reducing the number of $N$ that one needs to consider. As an illustration, in Section~\ref{sym sec}, we shall prove:

\begin{thm}\label{sym thm}Let $S_n$ be the symmetric group on $n$ letters, where $n\geq5$, and let $N$ be a group of order $n!$ with $e(S_n,N)\geq1$. Then, we have:
\begin{enumerate}[(1)]
\item $N$ fits into a short exact sequence $1\longrightarrow A_n\longrightarrow N\longrightarrow \{\pm1\}\longrightarrow 1$, or
\item $N$ fits into a short exact sequence $1\longrightarrow \{\pm1\}\longrightarrow N\longrightarrow A_n\longrightarrow 1$, or
\item $S_n$ embeds into $\mbox{Out}(N)$.
\end{enumerate}
Moreover, for any proper maximal characteristic subgroup $M$ of $N$, the quotient $N/M$ is either non-abelian or isomorphic to $\{\pm1\}$.
\end{thm}

Let us note that conditions $(1)$ and $(2)$ in Theorem~\ref{sym thm} cannot be removed because $e(S_n,S_n)\geq1$ and $e(S_n,A_n\times\{\pm1\})\geq1$ when $n\geq5$; see \cite{Childs simple} for the exact values of these two numbers.


\section{Regular subgroups of the holomorph}\label{Hol sec}

Throughout this section, let $G$ and $N$ denote two groups having the same order. Recall that given $\ff\in\Hom(G,\Aut(N))$, a map $\fg\in\Map(G,N)$ is said to be a \emph{crossed homomorphism with respect to $\ff$} if 
\[ \fg(\sigma_1\sigma_2) = \fg(\sigma_1)\cdot \ff(\sigma_1)(\fg(\sigma_2))\mbox{ for all }\sigma_1,\sigma_2\in G.\]
In general $\fg$ is not a group homomorphism, but for any $\sigma\in G$, we have
\begin{equation}\label{cyclic fg}\fg(\sigma^k) = \prod_{i=0}^{k-1}\ff(\sigma)^i(\fg(\sigma)) \mbox{ and in particular }\fg(\sigma^{e_\sigma k}) = \fg(\sigma^{e_\sigma})^k\end{equation}
for all $k\in\bN$, where $e_\sigma$ denotes the order of $\ff(\sigma)$. We shall write $Z_\ff^1(G,N)$ for the set of all such maps, and $Z_\ff^1(G,N)^*$ for the subset consisting of those which are bijective. 

\begin{prop}\label{corr thm}
For $\ff\in\Map(G,\Aut(N))$ and $\fg\in\Map(G,N)$, define
\[\beta_{(\ff,\fg)} : G\longrightarrow \Hol(N); \hspace{1em} \beta_{(\ff,\fg)}(\sigma) = \rho(\fg(\sigma))\cdot \ff(\sigma).\]
Then, we have
\begin{align*}
\Map(G,\Hol(N)) &= \{\beta_{(\ff,\fg)}:\ff\in\Map(G,\Aut(N))\mbox{ and }\fg\in\Map(G,N)\},\\
\Hom(G,\Hol(N)) &= \{ \beta_{(\ff,\fg)}: \ff\in\Hom(G,\Aut(N))\mbox{ and }\fg\in Z_\ff^1(G,N)\},\\
\Reg(G,\Hol(N)) &= \{ \beta_{(\ff,\fg)}: \ff\in\Hom(G,\Aut(N))\mbox{ and }\fg\in Z_\ff^1(G,N)^*\}.\end{align*}
\end{prop}
\begin{proof}The first equality is a direct consequence of (\ref{Hol(N)}). The second equality may be easily verified using the fact that
\[\rho(\eta_1)\varphi_1\cdot\rho(\eta_2)\varphi_2 = \rho(\eta_1)\varphi_1\rho(\eta_2)\varphi_1^{-1}\cdot\varphi_1\varphi_2 = \rho(\eta_1\varphi_1(\eta_2))\cdot\varphi_1\varphi_2\]
for $\eta_1,\eta_2\in N$ and $\varphi_1,\varphi_2\in\Aut(N)$. The third equality is then clear because 
\[ (\beta_{(\ff,\fg)}(\sigma))(1_N) = (\rho(\fg(\sigma))\cdot\ff(\sigma))(1_N) = \rho(\fg(\sigma))(1_N) =  \fg(\sigma)^{-1}\]
for $\sigma\in G$. This proves the proposition.
\end{proof}



In the rest of this section, let $\ff\in\Hom(G,\Aut(N))$ be fixed, and we shall consider some examples of $\fg\in Z_\ff^1(G,N)$.

\subsection{The trivial action}\label{trivial sec}Let $\ff_0\in\Hom(G,\Aut(N))$ be the trivial homomorphism, and note that $Z_{\ff_0}^1(G,N) = \Hom(G,N)$. For $N\not\simeq G$, we then deduce that $Z_{\ff_0}^1(G,N)^* = \emptyset$. As for $N=G$, we easily see from (\ref{Hol(N)}) that
\begin{equation}\label{must rho}
\mbox{for $\fg\in Z_\ff^1(G,G)^*:$}\hspace{1em} \beta_{(\ff,\fg)}(G) = \rho(G)\mbox{ if and only if }\ff = \ff_0.\end{equation}
Hence, the case when $\ff = \ff_0$ only gives rise to the regular subgroup $\rho(G)$.

\subsection{Principal crossed homomorphisms} Given any $\eta\in N$, it is natural to consider its associated \emph{principal} crossed homomorphism, defined by
\[ \fg_\eta\in Z_\ff^1(G,N); \hspace{1em}\fg_\eta(\sigma) = \eta^{-1}\cdot\ff(\sigma)(\eta).\]
Unfortunately, this map is never bijective, unless $G$ and $N$ are trivial. Indeed, viewing $N$ as a $G$-set via the homomorphism $\ff$, it is easy to check that
\[ \fg_\eta\mbox{ is injective if and only if }\Stab_{G}(\eta) = \{1_G\}.\]
In this case, since $|G| = |N|$, by the orbit-stabilizer theorem, we must have
\[ \{\ff(\sigma)(\eta) : \sigma\in G\} = N,\mbox{ and so }\ff(\sigma)(\eta) = 1_N\mbox{ for some }\sigma\in G.\]
This implies that $\eta = 1_N$, but $\fg_{1_N}$ is not bijective unless $G$ and $N$ are trivial.

\subsection{Action via inner automorphisms}For $\eta\in N$, we shall write
\[\conj(\eta) = \rho(\eta)\lambda(\eta)\mbox{ as well as }\conj(\eta Z(N)) = \conj(\eta).\]
The latter is plainly well-defined, and note that $\Inn(N)\simeq N/Z(N)$ via $\conj$. In the case that $\ff(G)\subset\Inn(N)$, elements in $Z_\ff^1(G,N)^*$ turn out to be closely related to certain \emph{fixed point free} pairs of homomorphisms. This connection was first observed by N. P. Byott and L. N. Childs in \cite{Byott Childs}; see the discussion at the end of Section~\ref{liftable sec}.

\begin{definition}\label{fpf pair def}A pair $(f,g)$, where $f,g\in\Hom(G,N)$, is \emph{fixed point free} if
\[ f(\sigma) = g(\sigma) \mbox{ holds precisely when }\sigma = 1_G.\]
Since $|G| = |N|$, this is easily seen to be equivalent to that the map $G\longrightarrow N$ given by $\sigma\mapsto f(\sigma)g(\sigma)^{-1}$ is bijective; see \cite[Proposition 1]{Byott Childs}, for example.
\end{definition}

We shall further make the following definition.

\begin{definition}\label{wfpf def}A pair $(f,g)$, where $f,g\in\Hom(G,N/Z(N))$, is \emph{weakly fixed point free} if the map $G\longrightarrow N/Z(N)$ given by $\sigma\mapsto f(\sigma)g(\sigma)^{-1}$ is surjective.
\end{definition}

\subsubsection{Liftable inner actions}\label{liftable sec} In what follows, assume that
\[\label{liftable hyp}\mbox{there exists $f\in\Hom(G,N)$ with }\ff(\sigma) = \conj(f(\sigma))\mbox{ for all $\sigma\in G$}.\]
This implies that $\ff(G)\subset \Inn(N)$ but the converse is false in general. Put
\[\Hom_f(G,N)^* = \{g\in\Hom(G,N): (f,g)\mbox{ is fixed point free}\}.\]
Then, we have:

\begin{prop}\label{fpf corr}The maps
\begin{align}\label{fpf map1}Z_\ff^1(G,N)& \longrightarrow \Hom(G,N);&\hspace{-2cm}\fg\mapsto (\sigma\mapsto \fg(\sigma)f(\sigma))\\\notag
Z_\ff^1(G,N)^* &\longrightarrow \Hom_f(G,N)^*;&\hspace{-2cm}\fg\mapsto (\sigma\mapsto \fg(\sigma)f(\sigma))
\end{align}
are well-defined bijections.
\end{prop}
\begin{proof}First, let $\fg\in Z_\ff^1(G,N)$. Then, for $\sigma_1,\sigma_2\in G$, we have
\begin{align*}
\fg(\sigma_1\sigma_2)f(\sigma_1\sigma_2)
& = \fg(\sigma_1)\conj(f(\sigma_1))(\fg(\sigma_2))\cdot f(\sigma_1)f(\sigma_2)\\\notag
& = \fg(\sigma_1)f(\sigma_1)\cdot \fg(\sigma_2)f(\sigma_2).
\end{align*}
This shows that the first map is well-defined. Next, let $g\in \Hom(G,N)$, and define $\fg(\sigma) = g(\sigma)f(\sigma)^{-1}$. Then, for $\sigma_1,\sigma_2\in G$, we have
\begin{align*}
\fg(\sigma_1\sigma_2) & = g(\sigma_1\sigma_2)f(\sigma_1\sigma_2)^{-1}\\
& = g(\sigma_1)f(\sigma_1)^{-1}\conj(f(\sigma_1))(g(\sigma_2)f(\sigma_2)^{-1})\\
& = \fg(\sigma_1)\cdot\ff(\sigma_1)(\fg(\sigma_2)).
\end{align*}
This shows the first map, which is plainly injective, is also surjective. From Definition~\ref{fpf pair def}, it is clear that $\fg$ is bijective if and only if $(f,g)$ is fixed point free. Hence, the second map is also a well-defined bijection.
\end{proof}

Let $\fg\in Z_\ff^1(G,N)$ and let $g\in\Hom(G,N)$ be its image under (\ref{fpf map1}). Then, for any $\sigma\in G$, we may rewrite
\[\beta_{(\ff,\fg)}(\sigma) = \rho(g(\sigma)f(\sigma)^{-1})\cdot\conj(f(\sigma)) = \rho(g(\sigma))\lambda(f(\sigma)) = \beta_{(f,g)}(\sigma),\]
where $\beta_{(f,g)}$ is the homomorphism defined as in \cite[Sections 2 and 3]{Byott Childs}. Hence, we may view Propositions~\ref{corr thm} and~\ref{fpf corr} as generalizations of \cite[Corollary 7]{Byott Childs}.

\subsubsection{General inner actions}\label{gen inn}  In what follows, assume that
\[\mbox{there exists $f\in\Hom(G,N/Z(N))$ with }\ff(\sigma) = \conj(f(\sigma))\mbox{ for all $\sigma\in G$}.\]
This implies that $\ff(G)\subset \Inn(N)$ and the converse is also true. Put
\begin{align*}\Hom_f(G,N/Z(N))^*& = \{g\in\Hom(G,N/Z(N)): (f,g)\mbox{ is}\\
&\hspace{3.5cm}\mbox{weakly fixed point free}\}.\end{align*}
Then, essentially the same argument as in Proposition~\ref{fpf corr} shows that:

\begin{prop}\label{wfpf corr}The maps
\begin{align}\label{fpf map2} Z_\ff^1(G,N)& \longrightarrow \Hom(G,N/Z(N));&\hspace{-0.5cm}\fg\mapsto (\sigma\mapsto \fg(\sigma)Z(N)\cdot f(\sigma))\\\notag
Z_\ff^1(G,N)^* &\longrightarrow \Hom_f(G,N/Z(N))^*;&\hspace{-0.5cm}\fg\mapsto (\sigma\mapsto \fg(\sigma)Z(N)\cdot f(\sigma))
\end{align}
are well-defined.
\end{prop}

However, the maps in Proposition~\ref{wfpf corr}, unlike those in Proposition~\ref{fpf corr}, are neither injective nor surjective in general.

\vspace{1.5mm}

Let $\fg\in Z_\ff^1(G,N)$ and let $g\in \Hom_f(G,N/Z(N))$ be its image under (\ref{fpf map2}). For any $\sigma\in G$, letting $\widetilde{f}(\sigma)\in N$ be an element such that $f(\sigma) = \widetilde{f}(\sigma)Z(N)$, we may then rewrite
\[ \beta_{(\ff,\fg)}(\sigma) = \rho(\fg(\sigma))\cdot\conj(\widetilde{f}(\sigma)) = \rho(\fg(\sigma)\widetilde{f}(\sigma))\lambda(\widetilde{f}(\sigma)).\]
Observe that $\rho(\eta) = \lambda(\eta)^{-1}$ for $\eta\in Z(N)$, and let $g_0\in \Hom(G,N/Z(N))$ be the trivial homomorphism. Then, we see that $\beta_{(\ff,\fg)}(G)\subset \lambda(N)$ when $g=g_0$. Thus, for $N\not\simeq G$, we have $g\neq g_0$ whenever $\fg$ is bijective. As for $N = G$, it is also easy to verify that
\begin{equation}\label{must lambda}
\mbox{for $\fg\in Z_\ff^1(G,G)^*:$}\hspace{1em}\beta_{(\ff,\fg)}(G) = \lambda(G)\mbox{ if and only if }g = g_0.
\end{equation}
 This is analogous to the discussion in Section~\ref{trivial sec}.

\section{Applications: isomorphic type}\label{app sec1}


\subsection{Abelian groups}\label{abelian sec} Let $A$ be an abelian group. 

\subsubsection{Backward implication} Suppose that
\[\label{|A| exp} |A| = 2^\delta p_1\cdots p_m,\mbox{ where $p_1<\cdots <p_m$ are odd primes and }\delta\in\{0,1,2\}.\]
To prove the backward implication of Theorem~\ref{abelian thm}, consider
\[ \ff \in \Hom(A,\Aut(A)) \mbox{ and } \fg\in Z_\ff^1(A,A).\]
By (\ref{Byott formula}) and Proposition~\ref{corr thm}, it is enough to show that
\begin{equation}\label{cyclic thm'}\beta_{(\ff,\fg)}(A) = \rho(A)\mbox{ whenever $\fg$ is bijective}.\end{equation}
Notice that the hypothesis on $A$ implies $A = A_0\times A_\sigma$, where $A_\sigma = \langle\sigma\rangle$ with $\sigma\in A$ an element of order $n = p_1\cdots p_m$, and $A_0$ is isomorphic to one of
\begin{equation}\label{A0} \{1\},\, \bZ/2\bZ,\, \bZ/4\bZ,\, \bZ/2\bZ\times \bZ/2\bZ.\end{equation}
Since $A_0$ and $A_\sigma$ are characteristic subgroups of $A$, their preimages $\fg^{-1}(A_0)$ and $\fg^{-1}(A_\sigma)$ are subgroups of $A$ as well. We then deduce that 
\begin{equation}\label{g(A)}
\fg(A_0) = A_0\mbox{ and }\fg(A_\sigma) = A_\sigma\mbox{ whenever $\fg$ is bijective}.\end{equation}
Observe that $\Aut(A) = \Aut(A_0)\times \Aut(A_\sigma)$. We have the following lemmas.

\begin{lem}\label{squarefree lem1}If $\fg$ is bijective, then $\ff(A_0)|_{A_0} = \{\mbox{Id}_{A_0}\}$.
\end{lem}
\begin{proof}For $|A_0|\leq 2$, it is clear that $\ff(A_0)|_{A_0} = \{\mbox{Id}_{A_0}\}$. For $|A_0|=4$, suppose that $\fg$ is bijective, and on the contrary that $\ff(\sigma_0)|_{A_0}\neq\mbox{Id}_{A_0}$ for some $\sigma_0\in A_0$.

\vspace{1.5mm}

First, assume that $A_0\simeq \bZ/4\bZ$. Without loss of generality, we may assume that $\sigma_0\in A_0$ is a generator. Note that $\ff(\sigma_0)(\tau) = \tau^{-1}$ for all $\tau\in A_0$, and so
\[ \fg(\sigma_0^2) = \fg(\sigma_0)\cdot \ff(\sigma_0)(\fg(\sigma_0)) = \fg(\sigma_0)\cdot\fg(\sigma_0)^{-1} = 1_{A_0},\]
by (\ref{g(A)}). But this contradicts that $\fg$ is bijective since $\sigma_0$ has order four.

\vspace{1.5mm}

Next, assume that $A_0\simeq \bZ/2\bZ\times \bZ/2\bZ$. Under this identification, applying a change of basis if necessary, we may assume that
\[\ff(\sigma_0)|_{A_0} = \begin{pmatrix} 0 & 1 \\ 1 &0\end{pmatrix},\mbox{ and so }\begin{pmatrix} 0 \\ 0\end{pmatrix} =  \fg(2\sigma_0) =  \fg(\sigma_0) + \begin{pmatrix} 0 & 1 \\ 1 &0\end{pmatrix}\fg(\sigma_0)\]
by (\ref{g(A)}). We must then have $\fg(\sigma_0) = (1,1)$. Now, again by (\ref{g(A)}), there exists $\tau_0\in A_0$ such that $\fg(\tau_0) = (1,0)$. But then
\[\fg(\sigma_0+\tau_0) = \fg(\sigma_0)+ \ff(\sigma_0)(\fg(\tau_0)) = \begin{pmatrix}1\\1 \end{pmatrix} + \begin{pmatrix} 0&1\\1&0\end{pmatrix}\begin{pmatrix}1\\0 \end{pmatrix}= \fg(\tau_0).\]
This contradicts that $\fg$ is bijective since plainly $\sigma_0$ is not the identity.
\end{proof}

\begin{lem}\label{squarefree lem2}If $\fg$ is bijective, then $\ff(A_\sigma)|_{A_\sigma} = \{\mbox{Id}_{A_\sigma}\}$.
\end{lem}
\begin{proof}Let $d\in\bZ$ be coprime to $n$ such that $\ff(\sigma)(\tau) = \tau^d$ for all $\tau\in A_\sigma$, and let $e$ denote the multiplicative order of $d$ mod $n$. For each $1\leq i\leq m$, write $e_i$ for the multiplicative order of $d$ mod $p_i$. Then, we have $e = \mbox{lcm}(e_1,\dots, e_m)$. Since $e$ divides $n$, we may also write
\[ e = p_{i_1}\cdots p_{i_r}\mbox{ for some }1\leq i_1<\cdots<i_r\leq m,\]
where the product may be empty. We then deduce from (\ref{cyclic fg}) and (\ref{g(A)}) that
\[ \fg(\sigma^{e(d-1)}) = \fg(\sigma)^{(1+d+\cdots + d^{e-1}) (d-1)} = \fg(\sigma)^{d^e-1} =  1_{A_\sigma}.\]
Now, suppose that $\fg$ is bijective. The above implies that
\begin{equation}\label{ed cond}e(d-1) \equiv 0 \hspace{-2mm}\pmod{n}.\end{equation}
We shall use this to show that $\{i_1,\dots,i_r\}$ is in fact empty.

\vspace{1.5mm}

For $i\notin\{i_1,\dots,i_r\}$, we have $e_i=1$ by (\ref{ed cond}). For $i\in\{i_1,\dots,i_r\}$, we have 
\[d^{e/p_i} \equiv (d^{p_i})^{e/p_i}\equiv 1\hspace{-2mm}\pmod{p_i},\]
whence $e_i$ divides both $e/p_i$ and $p_i-1$. This implies that $e_i$ divides $p_{i_1}\cdots p_{i_{j-1}}$ when $i = i_j$. But then $p_{i_r}$ cannot divide $e$, which is a contradiction. It follows that $\{i_1,\dots,i_r\}$ must be empty. This shows that $e=1$ and so $\ff(\sigma)|_{A_\sigma} = \mbox{Id}_{A_\sigma}$, as claimed.
\end{proof}

\begin{proof}[Proof of Theorem~\ref{abelian thm}: backward implication] Suppose that $\fg$ is bijective. To prove (\ref{cyclic thm'}), by (\ref{must rho}) as well as Lemmas~\ref{squarefree lem1} and~\ref{squarefree lem2}, it suffices to show that
\begin{equation}\label{abelian thm'}\ff(A_\sigma)|_{A_0} = \{\mbox{Id}_{A_0}\}\mbox{ and }\ff(A_0)|_{A_\sigma} = \{\mbox{Id}_{A_\sigma}\}.\end{equation}
For any $\sigma_0\in A_0$ and $\tau\in A_\sigma$, we have
\[ \fg(\sigma_0)\cdot \ff(\sigma_0)(\fg(\tau)) = \fg(\sigma_0\tau) = \fg(\tau\sigma_0) = \fg(\tau)\cdot \ff(\tau)(\fg(\sigma_0)). \]
From (\ref{g(A)}), we then deduce that
\[ \fg(\sigma_0) = \ff(\tau)(\fg(\sigma_0)) \mbox{ and }\fg(\tau) = \ff(\sigma_0)(\fg(\tau)),\]
and in particular (\ref{abelian thm'}) indeed holds.
\end{proof}

\subsubsection{Forward implication}  Suppose that $|A|$ does not have the form stated in Theorem~\ref{abelian thm}. This means that $A = H\times H'$, where $H$ is a subgroup of $A$ isomorphic to one of the groups in the next lemma.

\begin{lem}\label{H lem}Suppose that $H$ is isomorphic to one of the following:
\begin{enumerate}[(1)]
\item $\bZ/p^n\bZ$ or $\bZ/p\bZ\times\bZ/p\bZ$, where $p$ is an odd prime and $n\geq2$, or
\item $\bZ/2^n\bZ$, where $n\geq3$, or 
\item $\bZ/2\bZ\times\bZ/2\bZ\times\bZ/2\bZ$ or $\bZ/2\bZ\times \bZ/4\bZ$, or
\item $\bZ/4\bZ\times \bZ/4\bZ$.
\end{enumerate}
Then, there is a regular subgroup of $\Hol(H)$ which is isomorphic to $H$ but is not equal to $\rho(H)$.
\end{lem}
\begin{proof}For cases $(1)$ and $(2)$, see \cite[Lemmas 1 and 2]{By96} as well as its corrigendum. For case $(3)$, see \cite[Theorem 1.2.5]{Z thesis}. For case $(4)$, let us identify $H$ with $\bZ/4\bZ\times \bZ/4\bZ$, and define
\[ A_1 = \begin{pmatrix}1 & 0 \\ 0 & 1\end{pmatrix},\, A_2 = \begin{pmatrix} 3 & 2 \\ 2 & 3\end{pmatrix},\ \mathbf{b}_1 = \begin{pmatrix}1\\1\end{pmatrix},\, \mathbf{b}_2 = \begin{pmatrix}1\\2\end{pmatrix}.\]
Further, define two permutations $\eta_1,\eta_2$ on $H$ by setting
\[ \eta_1(\mathbf{x}) = A_1\mathbf{x} + \mathbf{b}_1\mbox{ and }\eta_2(\mathbf{x}) = A_2\mathbf{x} + \mathbf{b}_2\mbox{ for }\mathbf{x}\in H.\]
Note that $\eta_1,\eta_2\in \Hol(H)$ by (\ref{Hol(N)}). It is easy to verify that $\langle\eta_1,\eta_2\rangle\simeq H$ and $\langle\eta_1,\eta_2\rangle\neq\rho(H)$. A routine calculation also shows that $\langle\eta_1,\eta_2\rangle$ is regular, and so the claim follows.
\end{proof}

\begin{proof}[Proof of Theorem~\ref{abelian thm}: forward implication] By Lemma~\ref{H lem}, there is a regular subgroup $N_H$ of $\Hol(H)$ isomorphic to $H$ and not equal to $\rho(H)$. The image of $N_H\times \rho(H')$ under the natural injective homomorphism
\[ \Hol(H)\times \Hol(H') \longrightarrow \Hol(H\times H')\]
is then a regular subgroup isomorphic to $H\times H'$ and not equal to $\rho(H\times H')$. The claim now follows from (\ref{Byott formula}).
\end{proof}

\subsection{Quasisimple groups}\label{QS sec} Let $Q$ be a quasisimple group. We shall need the next proposition, which is the crucial ingredient for the proof of $e(Q,Q)= 2$ given in \cite{Childs simple} in the special case when $Q$ is non-abelian simple.

\begin{prop}\label{simple prop}Let $T$ be a non-abelian simple group. Then:
\begin{enumerate}[(a)]
\item Schreier's conjecture is true, namely $\Out(T)$ is solvable.
\item A pair $(f,g)$ with $f,g\in \Aut(T)$ is never fixed point free.
\end{enumerate}
\end{prop}
\begin{proof}They are consequences of the classification of finite simple groups; see \cite[Theorems 1.46 and 1.48]{G book}, for example.
\end{proof}

We shall also need the following basic properties concerning $Q$.

\begin{prop}\label{QS prop}The following statements hold.
\begin{enumerate}[(a)]
\item A proper normal subgroup of $Q$ is contained in $Z(Q)$.
\item The kernel of a non-trivial homomorphism $Q\longrightarrow Q/Z(Q)$ is $Z(Q)$.
\item The natural homomorphism $\Aut(Q)\longrightarrow \Aut(Q/Z(Q))$ is injective.
\item An endomorphism on $Q$ is either trivial or an automorphism.
\end{enumerate}
\end{prop}
\begin{proof}This is known, and it is easy to prove (a), which in turn gives (b). For the convenience of reader, we shall give a proof for (c) and (d). 

\vspace{1.5mm}

To prove (c), let $\varphi\in\Aut(Q)$ be such that $\varphi(\sigma)\sigma^{-1}\in Z(Q)$ for all $\sigma\in Q$. We easily verify that the map
\[\psi: Q\longrightarrow Z(Q);\hspace{1em}\psi(\sigma) = \varphi(\sigma)\sigma^{-1}\]
is a homomorphism. But then $\psi$ must be trivial because $Z(Q)$ is abelian and $Q=[Q,Q]$. This means that $\varphi = \mbox{Id}_Q$, as desired.

\vspace{1.5mm}

To prove (d), let $\varphi\in\Hom(Q,Q)$ be non-trivial, so then $\ker(\varphi)\subset Z(Q)$ by (a). Put $H =  \varphi(Q)$ for brevity. Note that
\[ \frac{Q/\ker(\varphi)}{Z(Q)/\ker(\varphi)} \simeq \frac{Q}{Z(Q)}\mbox{ has trivial center}\]
as well as that $Q/\ker(\varphi)\simeq H$. Hence, we deduce that
\[Z\left(\frac{Q}{\ker(\varphi)}\right) \subset \frac{Z(Q)}{\ker(\varphi)}\mbox{ and so }|Z(H)| \leq [Z(Q):\ker(\varphi)].\]
It follows that
\[ |HZ(Q)| = \frac{|H||Z(Q)|}{|H\cap Z(Q)|} \geq \frac{|H||Z(Q)|}{|Z(H)|} \geq \frac{|H||Z(Q)|}{[Z(Q):\ker(\varphi)]} = |Q|,\]
and so $HZ(Q) = Q$. Since $Q = [Q,Q]$, we deduce that $H = Q$. This means that $\varphi\in\Aut(Q)$, as desired.
\end{proof}

Now, to prove Theorem~\ref{QS thm}, consider
\[\ff\in\Hom(Q,\Aut(Q))\mbox{ and }\fg\in Z_\ff^1(Q,Q).\]
By (\ref{Byott formula}) and Proposition~\ref{corr thm}, it suffices to show that
\begin{equation}\label{QS thm'} \beta_{(\ff,\fg)}(Q) \in \{\rho(Q),\lambda(Q)\}\mbox{ whenever $\fg$ is bijective}.\end{equation}
The next lemma is analogous to an argument in \cite[p. 84]{Childs simple}.

\begin{lem}\label{QS inn}We have $\ff(Q) \subset \Inn(Q)$.
\end{lem}
\begin{proof}The group $\Out(Q/Z(Q))$ is solvable by Proposition~\ref{simple prop} (a). Since we have an injective homomorphism 
\[\Out(Q)\longrightarrow\Out(Q/Z(Q))\]
by Proposition~\ref{QS prop} (c), the group $\Out(Q)$ is also solvable. Since $Q = [Q:Q]$, we then see that the homomorphism
\[\begin{tikzcd}[column sep = 1.5cm] Q\arrow{r}{\ff}& \Aut(Q) \arrow{r}{\mbox{\tiny quotient}} &\Out(Q)\end{tikzcd}\]
must be trivial. This means that $\ff(Q)\subset\Inn(Q)$, as claimed.
\end{proof}

In view of Lemma~\ref{QS inn}, we may define
\[ f\in \Hom(Q,Q/Z(Q)),\, \widetilde{f}\in\Map(Q,Q),\, g\in\Hom(Q,Q/Z(Q))\]
as in Subsection~\ref{gen inn}. More precisely, we have
\[ \ff(\sigma) = \conj(f(\sigma)),\,\ f(\sigma) = \widetilde{f}(\sigma)Z(Q),\,\ g(\sigma) = \fg(\sigma)\widetilde{f}(\sigma)Z(Q)\]
for all $\sigma\in Q$. We make the following useful observation.

\begin{lem}\label{QS gZ}If $f$ and $g$ are non-trivial, then $\fg(Z(Q))\subset Z(Q)$.
\end{lem}
\begin{proof}Let $\sigma\in Z(Q)$. Then, for any $\tau\in Q$, we have
\[\fg(\sigma)\widetilde{f}(\sigma)\fg(\tau)\widetilde{f}(\sigma)^{-1} = \fg(\sigma\tau) = \fg(\tau\sigma) = \fg(\tau)\widetilde{f}(\tau)\fg(\sigma)\widetilde{f}(\tau)^{-1}.\]
Suppose that $f$ is non-trivial, so then $\ker(f) = Z(Q)$ by Proposition~\ref{QS prop} (b). This implies that $\widetilde{f}(\sigma) \in Z(Q)$, and from the above, we deduce that $\fg(\sigma)$ is centralized by $\fg(\tau)\widetilde{f}(\tau)$. This means that $\fg(\sigma)$ commutes with
\[ \bigcup_{\tau\in Q}\fg(\tau)\widetilde{f}(\tau)Z(Q) = \bigcup_{\tau\in Q}g(\tau).\]
Now, suppose further that $g$ is non-trivial. Then, by Proposition~\ref{QS prop} (b), we\par\noindent have $\ker(g) = Z(Q)$, and in particular $g$ is surjective. Hence, the union above is equal to the entire group $Q$. This means that $\fg(\sigma)\in Z(Q)$, as desired.
\end{proof}



\begin{proof}[Proof of Theorem~\ref{QS thm}] Suppose that $\fg$ is bijective. In view of (\ref{must rho}) and (\ref{must lambda}), to prove (\ref{QS thm'}), it suffices to show that either $f$ or $g$ is trivial. Suppose for contradiction that they are both non-trivial. Then, they induce automorphisms
\[\overline{f} : Q/Z(Q)\longrightarrow Q/Z(Q)\mbox{ and }\overline{g}: Q/Z(Q)\longrightarrow Q/Z(Q)\]
by Proposition~\ref{QS prop} (b). For any $\sigma\in Q$, we have
\[ f(\sigma) = g(\sigma) \implies \fg(\sigma)\in Z(Q) \implies \sigma\in Z(Q),\]
by Lemma~\ref{QS gZ} and the bijectivity of $\fg$. This shows that $(\overline{f},\overline{g})$ is fixed point free, which is impossible by Proposition~\ref{simple prop} (b). 
\end{proof}

\section{Applications: non-isomorphic type}\label{app sec2}


\subsection{Formulation of the main idea}\label{idea sec}In what follows, let $G$ and $N$ denote two groups having the same order. Notice that by definition, any characteristic subgroup $M$ of $N$ is normal, and we also have a natural homomorphism $\Aut(N)\longrightarrow \Aut(N/M)$.

\begin{lem}\label{mod char}Let $M$ be a characteristic subgroup of $N$. Given
\[ \ff\in\Hom(G,\Aut(N)) \mbox{ and }\fg\in Z_\ff^1(G,N),\]
they induce, respectively, a homomorphism and a map
\begin{equation}\label{fg bar} \overline{\ff}:G\longrightarrow\Aut(N)\longrightarrow\Aut(N/M)\mbox{ and }\overline{\fg}:G\longrightarrow N\longrightarrow N/M.\end{equation}
By abuse of notation, define
\[ \ker(\overline{\fg}) = \{\sigma\in G : \overline{\fg}(\sigma) = 1_{N/M}\}.\]
Then, the following are true.
\begin{enumerate}[(a)]
\item The set $\ker(\overline{\fg})$ is a subgroup of $G$.
\item The map $\overline{\fg}$ induces an injection $G/\ker(\overline{\fg})\longrightarrow N/M$.
\item The map $\overline{\fg}$ restricts to a homomorphism $\ker(\overline{\ff})\longrightarrow N/M$.
\item In the case that $N/M$ is abelian, for any $\sigma\in \ker(\overline{\ff})\cap Z(G)$, the element $\overline{\fg}(\sigma)$ is fixed by the automorphisms in $\overline{\ff}(G)$.
\end{enumerate}
\end{lem}
\begin{proof}Both (a) and (c) are clear. For (b), simply observe that
\begin{eqnarray*}
\overline{\fg}(\sigma_1) = \overline{\fg}(\sigma_2)
& \iff& \overline{\fg}(\sigma_1^{-1}\sigma_2) = \overline{\fg}(\sigma_1^{-1})\cdot\overline{\ff}(\sigma_1^{-1})(\overline{\fg}(\sigma_1))\\
&\iff& \overline{\fg}(\sigma_1^{-1}\sigma_2) = \overline{\fg}(\sigma_1^{-1}\sigma_1)\\
&\iff&\overline{\fg}(\sigma_1^{-1}\sigma_2) = 1_{N/M}
\end{eqnarray*}
for any $\sigma_1,\sigma_2\in G$. Finally, statement (d) follows from the fact that
\[ \overline{\fg}(\sigma)\overline{\fg}(\tau) = \overline{\fg}(\sigma)\cdot\overline{\ff}(\sigma)(\overline{\fg}(\tau)) = \overline{\fg}(\sigma\tau) = \overline{\fg}(\tau\sigma)  = \overline{\fg}(\tau)\cdot\overline{\ff}(\tau)(\overline{\fg}(\sigma))\]
for any $\tau\in G$ and $\sigma\in\ker(\overline{\ff})\cap Z(G)$.
\end{proof}

We keep the notation as in Lemma~\ref{mod char}. To show that $e(G,N)=0$, by (\ref{Byott formula}) and Proposition~\ref{corr thm}, it is the same as proving that $\fg$ can never be bijective. The idea is that, while we might not understand $\Aut(N)$ or $N$ very well, by passing to $\Aut(N/M)$ and $N/M$ for a suitable characteristic subgroup $M$ of $N$, we might be able to use Lemma~\ref{mod char} to prove the weaker statement that $\overline{\fg}$ can never be surjective.

\subsection{Cyclic groups of odd prime power order}\label{pn sec} Let $C_{p^n}$ be a cyclic group of odd prime power order $p^n$ and let $N$ be a group of order $p^n$ with $N\not\simeq C_{p^n}$. To prove Theorem~\ref{cyclic thm}, consider
\[\ff\in \Hom(C_{p^n},\Aut(N))\mbox{ and }\fg\in Z_\ff^1(C_{p^n},N).\]
By (\ref{Byott formula}) and Proposition~\ref{corr thm}, it is enough to show that $\fg$ cannot be bijective. Take $M$ to be the Frattini  subgroup $\Phi(N)$ of $N$. Then, we know that
\[ N/\Phi(N) \simeq (\bZ/p\bZ)^m\mbox{ and so }\Aut(N/\Phi(N)) \simeq \GL_m(\bZ/p\bZ),\]
where $m\in\bN$ is such that $m\geq2$ because $N$ is non-cyclic. Let $\overline{\ff}$ and $\overline{\fg}$ be as in (\ref{fg bar}). Then, in turn, it suffices to show that $\overline{\fg}$ cannot be surjective.

\vspace{1.5mm}

Fix a generator $\sigma\in C_{p^n}$, and write $|\overline{\ff}(\sigma)| = p^r$, where $0\leq r\leq n$. The next two lemmas yield, respectively, an upper bound and a lower bound for $p^r$ in terms of $m$ and the index of $\ker(\overline{\fg})$ in $C_{p^n}$.

\begin{lem}\label{cyclic matrix lem}Let $B\in \GL_m(\bZ/p\bZ)$ be a matrix of order $p^r$.
\begin{enumerate}[(a)]
\item If $m\geq3$, then $r\leq m-2$.
\item If $m=2$, then $r\leq1$, and $B$ is conjugate to $\left(\begin{smallmatrix}1&0\\0&1\end{smallmatrix}\right)$ or $\left(\begin{smallmatrix}1&1\\0&1\end{smallmatrix}\right)$ in $\GL_2(\bZ/p\bZ)$.
\end{enumerate}
\end{lem}
\begin{proof}The fact that $B$ has order a power of $p$ implies that $B$ is conjugate to a Jordan matrix with $\lambda=1$ on the diagonal in $\GL_m(\bZ/p\bZ)$. From here, it is easy to see that $B^{p^{m-2}}$ and $B^p$, respectively, for $m\geq3$ and $m=2$, equal the identity matrix in $\GL_m(\bZ/p\bZ)$. The claim now follows.
\end{proof}

\begin{lem}\label{cyclic lem}The following statements hold.  
\begin{enumerate}[(a)]
\item If $m\geq3$, then $\langle \sigma^{p^{r+1}}\rangle\subset\ker(\overline{\fg})$ and so
$[C_{p^n}:\ker(\overline{\fg})] \leq p^{r+1}$. 
\item If $m=2$, then $\langle \sigma^p\rangle\subset\ker(\overline{\fg})$ and so $[C_{p^n}:\ker(\overline{\fg})]\leq p$.
\end{enumerate}
\end{lem}
\begin{proof}Note that $N/\Phi(N)$ has exponent $p$. By (\ref{cyclic fg}), we also have
\[ \overline{\fg}(\sigma^{p^{r+1}}) = \overline{\fg}(\sigma^{p^r})^p\mbox{ and }\overline{\fg}(\sigma^p) = \prod_{i=0}^{p-1}\overline{\ff}(\sigma)^i(\overline{\fg}(\sigma)).\]
The claim for $m\geq3$ then follows from the first equality. Now, suppose that $m=2$. Then, regarding $\overline{\ff}(\sigma)$ an element in $\GL_2(\bZ/p\bZ)$, by Lemma~\ref{cyclic matrix lem}, it is conjugate to a matrix $\left(\begin{smallmatrix}1&b\\0&1\end{smallmatrix}\right)$ with $b\in\bZ/p\bZ$. Since
\begin{equation}\label{p odd}
\sum_{i=0}^{p-1}\begin{pmatrix}1&b\\0&1 \end{pmatrix}^i  = \begin{pmatrix} p & \frac{p(p-1)b}{2} \\ 0 & p \end{pmatrix} = \mbox{zero matrix in $\GL_2(\bZ/p\bZ)$},\end{equation}
we see that indeed $\sigma^p\in\ker(\overline{\fg})$. This proves the claim.
\end{proof}

\begin{proof}[Proof of Theorem~\ref{cyclic thm}]Lemmas~\ref{cyclic matrix lem} and~\ref{cyclic lem} imply that 
\[[C_{p^n} : \ker(\overline{\fg})] \leq p^{m-1}.\]
We then see from Lemma~\ref{mod char} (b) that $\overline{\fg}$ indeed cannot be surjective.
\end{proof}

\begin{remark}Note that the hypothesis that $p$ is odd is required for the second equality in (\ref{p odd}). In fact, the analogous statement of Theorem~\ref{cyclic thm} for $p=2$ is false, as shown in \cite[Corollary 5.3]{Byott almost cyclic}.
\end{remark}

\subsection{Groups of order $n$ factorial} In what follows, let $n\in\bN$ with $n\geq5$. Recall that $2A_n$ is the unique group, up to isomorphism, fitting into a short exact sequence
\[\begin{tikzcd}[column sep = 1cm]1 \arrow{r}& \{\pm1\}\arrow{r}{\iota}&2A_n\arrow{r} &A_n\arrow{r}& 1\end{tikzcd}\]such that $\iota(\{\pm1\})$ lies in both $Z(2A_n)$ and $[2A_n,2A_n]$. It is known that $2A_n$ is quasisimple and $Z(2A_n)\simeq \{\pm1\}$. We then have:

\begin{lem}\label{2An unique}If $N = [N,N]$ and there is a normal subgroup $M$ of $N$ having order two such that $N/M\simeq A_n$, then necessarily $N \simeq 2A_n$.
\end{lem}
\begin{proof}This is because a normal subgroup of order two lies in the center.
\end{proof}

There are some similarities in the proofs of Theorems~\ref{2An thm} and~\ref{sym thm} because:
\begin{enumerate}[(i)]
\item Both $2A_n$ and $S_n$ have order $n!$. 
\item Both $2A_n$ and $S_n$ have only one non-trivial proper normal subgroup.
\item The alternating group $A_n$ is a subgroup of $S_n$ and is a quotient of $2A_n$.
\end{enumerate}
For (ii), we in particular know that
\begin{equation}\label{normal}\begin{cases}
Z(2A_n)\mbox{ is the non-trivial proper normal subgroup of }2A_n,\\
A_n\mbox{ is the non-trivial proper normal subgroup of }S_n,
\end{cases}\end{equation}
where the first statement follows from Proposition~\ref{QS prop} (a). Given a prime $p$ and a group $\Gamma$, write $v_p(\Gamma)$ for the non-negative integer such that 
\[p^{v_p(\Gamma)}=\mbox{the exact power of $p$ dividing $|\Gamma|$}.\]
Motivated by the arguments in \cite{Byott simple}, we shall require the next two lemmas.

\begin{lem}\label{An lem}If $A_n$ has a subgroup of prime power index $p^m$, then $n = p^m$.
\end{lem}
\begin{proof}See \cite[(2.2)]{RG}.
\end{proof}

\begin{lem}\label{val lem} Let $m\in\bN$ and let $p$ be a prime. Then, we have
\begin{enumerate}[(a)]
\item $|\GL_m(\bZ/p\bZ)| < \frac{1}{2}(p^m!)$ and $v_p(\GL_m(\bZ/p\bZ)) = m(m-1)/2$,
\item $v_p(S_m) < m$ and $v_p(S_{p^m}) \geq m(m+1)/2$,
\item $v_2(S_{2^{m-1}}) \geq m(m-1)/2 + 2$ for $m\geq 5$.
\end{enumerate}
\end{lem}
\begin{proof}Both claims in (a) and the first claim in (b) follow from
\[ |\GL_m(\bZ/p\bZ)| = \prod_{i=0}^{m-1}(p^m - p^i)\mbox{ and } v_p(S_m) = \sum_{i=1}^{\infty}\left\lfloor\frac{m}{p^i}\right\rfloor,\]
respectively, as in \cite[(4.1) and Lemma 3.3]{Byott simple}. The two remaining claims hold because $p^m!$ is divisible by $p\cdots p^m$, and $2^{m-1}!$ is divisible by $2\cdots 2^{m-1}\cdot 6\cdot 10$ for $m\geq 5$. 
\end{proof}

Now, for both Theorems~\ref{2An thm} and~\ref{sym thm}, let $N$ be a group of order $n!$, and let $M$ be any proper maximal characteristic subgroup of $N$. The quotient $N/M$ is then a non-trivial and \emph{characteristically simple} group, meaning that it has no non-trivial proper characteristic subgroup. It is then known that
\begin{equation}\label{N/M} N/M\simeq T^m,\mbox{ where $T$ is simple and $m\in\bN$}.\end{equation}
As shown in \cite[Lemma 3.2]{Byott simple}, we have
\[\Aut(N/M)\simeq \Aut(T)^m\rtimes S_m\mbox{ when $T$ is non-abelian}.\]
The structure of $\Aut(N/M)$ is of course well-understood when $T$ is abelian.

\subsubsection{The double cover of alternating groups}\label{2An sec} To prove Theorem~\ref{2An thm}, in this section, we shall assume that $N\not\simeq 2A_n$. Consider
\[ \ff \in \Hom(2A_n,\Aut(N)) \mbox{ and }\fg\in Z_\ff^1(2A_n,N).\]
By (\ref{Byott formula}) and Proposition~\ref{corr thm}, it is enough to show that $\fg$ cannot be bijective. We shall use the same notation as in (\ref{N/M}), and let $\overline{\ff},\overline{\fg}$ be as in (\ref{fg bar}). Then, in turn, it suffices to show that $\overline{\fg}$ cannot be surjective. 

\begin{lem}\label{2An lem1}If $N = [N:N]$, then $\overline{\ff}$ is trivial.
\end{lem}
\begin{proof}Suppose that $N = [N,N]$, in which case $T$ is non-abelian. Consider
\begin{equation}\label{proj Sm}\begin{tikzcd}[column sep = 1.5cm]
2A_n \arrow{r}{\overline{\ff}}& \Aut(N/M) \ar[equal]{r}{ \mbox{\tiny identification}}  & \Aut(T)^m\rtimes S_m \arrow{r}{\mbox{\tiny projection}} & S_m.
\end{tikzcd}\end{equation}
Notice that $|T|$ has an odd prime factor $p$ because a $2$-group has non-trivial center. We have $v_p(S_m) < m$ by Lemma~\ref{val lem} (b) while
\[ v_p(2A_n/Z(2A_n)) = v_p(2A_n) = v_p(N) = m\cdot v_p(T) + v_p(M) \geq m. \]
We then deduce from (\ref{normal}) that (\ref{proj Sm}) must be trivial. It follows that $\overline{\ff}(2A_n)$ lies in $\Aut(T)^m$. Note that the homomorphism
\[\begin{tikzcd}[column sep = 1.5cm]
2A_n \arrow{r}{\overline{\ff}} &\Aut(T)^m \arrow{r}{\mbox{\tiny projection}} & \Out(T)^m
\end{tikzcd}\]
must also be trivial, because $2A_n = [2A_n,2A_n]$ while $\Out(T)^m$ is solvable by Proposition~\ref{simple prop} (a). Hence, in fact $\overline{\ff}:2A_n\longrightarrow\Inn(T)^m\simeq T^m$.

\vspace{1.5mm}

Now, by (\ref{normal}), either $\overline{\ff}$ is trivial or $|\ker(\overline{\ff})| \leq 2$. Observe that
\[ [2A_n:\ker(\overline{\ff})] \leq |T|^m = [N:M] = |2A_n|/|M| \mbox{ and so }|M|\leq |\ker(\overline{\ff})|.\]
If $|\ker(\overline{\ff})|=1$, then $|M|= 1$, and we deduce that $2A_n\simeq T^m\simeq N$, which is a contradiction. If $|\ker(\overline{\ff})| = 2$ and $|M|=1$, then $\overline{\ff}(2A_n)$ has index two and in particular is normal in $T^m\simeq N$, but this is impossible because $N = [N,N]$. Finally, if $|\ker(\overline{\ff})| = 2$ and $|M|=2$, then $A_n\simeq\overline{\ff}(2A_n)\simeq T^m\simeq N/M$, which contradicts that $2A_n\not\simeq N$ by Lemma~\ref{2An unique}. Thus, indeed $\overline{\ff}$ must be trivial.
\end{proof}

\begin{lem}\label{2An lem2}If $N/M$ is abelian and $\overline{\fg}$ is surjective, then $\overline{\ff}$ is trivial. 
\end{lem}
\begin{proof}Recall that $2A_n/Z(2A_n)\simeq A_n$, and note that
\[ \left[\frac{2A_n}{Z(2A_n)}:\frac{\ker(\overline{\fg})Z(2A_n)}{Z(2A_n)}\right] = \begin{cases}
\hspace{2.5mm}[2A_n : \ker(\overline{\fg})] &\mbox{if $Z(2A_n)\subset \ker(\overline{\fg})$},\\
\frac{1}{2}[2A_n:\ker(\overline{\fg})] & \mbox{if $\ker(\overline{\fg})
\cap Z(2A_n)=1$}.
\end{cases}\]
These are the only cases because $Z(2A_n)$ has order two. 

\vspace{1.5mm}

Suppose that $N/M$ is abelian and $\overline{\fg}$ is surjective. Then, we have $T\simeq \bZ/p\bZ$ for some prime $p$, and $[2A_n:\ker(\overline{\fg})] = p^m$ by Lemma~\ref{mod char} (b).  We also have
\[\begin{cases}
n = p^m &\mbox{if $Z(2A_n)\subset \ker(\overline{\fg})$},\\
n = 2^{m-1}\mbox{ with }m\geq 4 &\mbox{if $\ker(\overline{\fg})
\cap Z(2A_n)=1$},
\end{cases}\]
by Lemma~\ref{An lem}. Recall Lemma~\ref{val lem}, and observe that
\[\begin{cases} [2A_{p^m}:Z(2A_{p^m})] > |\GL_m(\bZ/p\bZ)|\\
 v_2(2A_{2^{m-1}}/Z(2A_{2^{m-1}})) > v_2(\GL_m(\bZ/2\bZ))\mbox{ for $m\geq5$}.
\end{cases}\]
Since $2A_n/\ker(\overline{\ff})$ embeds into $\GL_m(\bZ/p\bZ)$, we then deduce from (\ref{normal}) that $\overline{\ff}$ \par\noindent must be trivial, except possibly when 
\begin{equation}\label{last case}\ker(\overline{\fg})\cap Z(2A_n)=1\mbox{ and }n=2^{m-1}\mbox{ for }m=4.\end{equation}
In this last case (\ref{last case}), suppose for contradiction that $\overline{\ff}$ is non-trivial. Since
\begin{equation}\label{m=4}|2A_8| = 40320\mbox{ and }|\GL_4(\bZ/2\bZ)| = 20160,\end{equation}
necessarily $\overline{\ff}$ is surjective and $\ker(\overline{\ff}) = Z(2A_8)$. For any $\sigma\in Z(2A_8)$, we then deduce from Lemma~\ref{mod char} (d) that $\overline{\fg}(\sigma)$ is a fixed point of every automorphism on $N/M\simeq (\bZ/2\bZ)^4$, and so $\overline{\fg}(\sigma) = 1_{N/M}$. This shows that $Z(2A_8)\subset \ker(\overline{\fg})$, which contradicts the first condition in (\ref{last case}).
\end{proof}

\begin{proof}[Proof of Theorem~\ref{2An thm}]Suppose for contradiction that $\overline{\fg}$ is surjective. In the case that $N\supsetneq[N,N]$, we may choose $M$ to be such that $M\supset [N,N]$, which ensures that $N/M$ is abelian. Then, by Lemmas~\ref{2An lem1} and~\ref{2An lem2}, we know that $\overline{\ff}$ is trivial, whence $\overline{\fg}$ is a homomorphism, and so we have $N/M\simeq 2A_n/\ker(\overline{\fg})$. Notice that $N/M$ is non-trivial because $M$ is proper by choice. By (\ref{normal}) and the hypothesis that $N\not\simeq 2A_n$, we then deduce that $\ker(\overline{\fg}) = Z(2A_n)$ and so $N/M\simeq A_n$. We now have a contradiction because:
\begin{itemize}
\item If $N\supsetneq[N,N]$, then $N/M$ is abelian by the choice of $M$.
\item If $N=[N,N]$, then $N\simeq 2A_n$ by Lemma~\ref{2An unique}, but $N\not\simeq 2A_n$ by hypothesis.
\end{itemize}
\vspace{1.5mm}
Thus, the map $\overline{\fg}$ cannot be surjective, and the theorem now follows.
\end{proof}

\subsubsection{Symmetric groups}\label{sym sec} To prove Theorem~\ref{sym thm}, consider
\[ \ff \in\Hom(S_n,\Aut(N))\mbox{ and }\fg\in Z_\ff^1(S_n,N).\]
To prove the first statement, by (\ref{Byott formula}) and Proposition~\ref{corr thm}, it suffices to show that one of the three stated conditions holds whenever $\fg$ is bijective.

\begin{proof}[Proof of Theorem~\ref{sym thm}: first statement] First, suppose that $\ker(\ff) = 1$. Since
\[ \ff(S_n) \cap \Inn(N) \mbox{ is a normal subgroup of }\ff(S_n)\simeq S_n,\]
by (\ref{normal}) we have $\ff(S_n)\cap \Inn(N) \in \{\ff(S_n),\ff(A_n),1\}$. It is easy to see that:
\begin{enumerate}[(i)]
\item If $\Inn(N)\cap \ff(S_n) = \ff(S_n)$, then $N\simeq S_n$ so condition $(1)$ holds.
\item If $\Inn(N)\cap\ff(S_n) = \ff(A_n)$, then $A_n$ embeds into $\Inn(N)\simeq N/Z(N)$ and thus $|Z(N)|\leq 2$. Then, condition $(1)$ holds when $|Z(N)|=1$ because a subgroup of index two is always normal, and condition $(2)$ clearly holds when $|Z(N)|=2$.
\item If $\Inn(N)\cap\ff(S_n) = 1$, then condition $(3)$ holds.
\end{enumerate}
Note that we do not need $\fg$ to be bijective for the above arguments.

\vspace{1.5mm}

Now, suppose that $\ker(\ff)\neq 1$, so then $\ker(\ff)\in\{A_n,S_n\}$ by (\ref{normal}). Suppose also that $\fg$ is bijective. If $\ker(\ff) = S_n$, then $N\simeq S_n$ by (\ref{must rho}). If $\ker(\ff) = A_n$, then $N$ contains a subgroup isomorphic to $A_n$ by Lemma~\ref{mod char} (c), which has index two and hence is normal in $N$. In both cases, we see that condition $(1)$ holds. This proves the first statement of the theorem.
\end{proof}

To prove the second statement, let $M$ be any proper maximal characteristic subgroup of $N$. We shall use the notation in (\ref{N/M}), and let $\overline{\ff},\overline{\fg}$ be as in (\ref{fg bar}). By (\ref{Byott formula}) and Proposition~\ref{corr thm}, it suffices to show that 
\[N/M\simeq \bZ/2\bZ\mbox{ whenever $N/M$ is abelian and $\overline{\fg}$ is surjective}.\]
The reader should compare our proof below with that of Lemma~\ref{2An lem2}.

\begin{proof}[Proof of Theorem~\ref{sym thm}: second statement]  Note that 
\[ [A_n: A_n\cap \ker(\overline{\fg})] = \begin{cases}
[S_n:\ker(\overline{\fg})] &\mbox{if }\ker(\overline{\fg})\not\subset A_n,\\
\frac{1}{2}[S_n:\ker(\overline{\fg})]&\mbox{if }\ker(\overline{\fg})\subset A_n,
\end{cases}\]
and these are the only cases because $[S_n:A_n]=2$. 

\vspace{1.5mm}

Suppose that $N/M$ is abelian and $\overline{\fg}$ is surjective. Then, we have $T\simeq \bZ/p\bZ$ for some prime $p$, and $[S_n:\ker(\overline{\fg})] = p^m$ by Lemma~\ref{mod char} (b).  We also have
\[\begin{cases}
n = p^m &\mbox{if }\ker(\overline{\fg})\not\subset A_n,\\
n = 2^{m-1}\mbox{ with }m\geq 4&\mbox{if }\ker(\overline{\fg})\subset A_n,
\end{cases}\] 
by Lemma~\ref{An lem}, as well as
\[\begin{cases} v_p(S_{p^m}) > v_p(\GL_m(\bZ/p\bZ)) &\mbox{for all $m\geq1$},\\
v_2(S_{2^{m-1}}) > v_2(\GL_m(\bZ/2\bZ))&\mbox{for all $m\geq4$},\end{cases}\]
by Lemma~\ref{val lem} and (\ref{m=4}). Since $S_n/\ker(\overline{\ff})$ embeds into $\GL_m(\bZ/p\bZ)$, we see that  $\ker(\overline{\ff})\neq1$, and so $\ker(\overline{\ff})\supset A_n$ by (\ref{normal}). It then follows from Lemma~\ref{mod char} (c) that $\overline{\fg}$ 
restricts to a homomorphism $A_n\longrightarrow N/M$. Since $A_n = [A_n,A_n]$ and $N/M$ is abelian, this implies that $A_n\subset\ker(\overline{\fg})$. But then
\[ 2 = [S_n : A_n] \geq [S_n:\ker(\overline{\fg})] = [N:M]\]
by Lemma~\ref{mod char} (b), and so we must have $N/M\simeq \bZ/2\bZ$, as claimed.
\end{proof}

\section{Acknowledgements} The author is partially supported by the China Postdoctoral Science Foundation Special Financial Grant (grant no.: 2017T100060). She would like to thank the University of Exeter for their hospitality when she presented this research there in May 2018. She would particularly like to thank Prof. N. P. Byott for showing interest in this work and for sending her the thesis \cite{Z thesis} of K. N. Zenouz cited in Lemma~\ref{H lem}. She would also like to thank the referee for pointing out some typos and unclear arguments in the original manuscript.

\end{document}